\documentclass[11pt]{amsart}
\usepackage{amssymb, latexsym, amsmath, amsfonts, tikz}
\usepackage{hyperref, enumitem, fancyhdr}
\usepackage{color}

\newtheorem{thm}{Theorem}[section]
\newtheorem{cor}[thm]{Corollary}
\newtheorem{lem}[thm]{Lemma}
\newtheorem{prop}[thm]{Proposition}
\theoremstyle{definition}
\newtheorem{defn}[thm]{Definition}
\theoremstyle{remark}

\numberwithin{equation}{section}
\theoremstyle{remark}

\newcommand{\no}{\noindent}

\newcommand{\bea}{\begin{eqnarray*}}
\newcommand{\eea}{\end{eqnarray*}}

\begin{document}
\title{ Notes on the Short $\mathbb{C}^k$'s}

\thanks{}
\author{John Erik Forn\ae ss and Ratna Pal}

\address{John Erik Forn\ae ss: Department of Mathematical Sciences, NTNU Trondheim, Norway}
\email{john.fornass@ntnu.no}

\address{Ratna Pal: Indian Institute of Science Education and Research Berhampur, Transit campus (Govt. ITI Building), Engg. School Junction, Berhampur, Odisha 760 010, India}
\email{ratnap@iiserbpr.ac.in} 

\begin{abstract}
Domains that are increasing union of balls (up to biholomorphism) and on which the Kobayashi metric vanishes identically arise inexorably in complex analysis. In this article we show that in higher dimensions these domains have infinite volume and the Bergman spaces of these domains are trivial. As a consequence they fail to be strictly pseudo-convex at each of their boundary points although these domains are pseudo-convex by definition.  These domains can be of different types and one of them is Short $\mathbb{C}^k$'s. In pursuit of identifying the Runge Short $\mathbb{C}^k$'s (up to biholomorphism), we introduce a special class of Short $\mathbb{C}^k$'s, called Loewner Short $\mathbb{C}^k$'s. These are those Short $\mathbb{C}^k$'s which can be exhausted in a continuous manner by a strictly increasing parametrized family of open sets, each of which is biholomrphically equivalent to the unit ball and therefore, they are Runge up to biholomorphism. Although, the question of whether all Short $\mathbb{C}^k$'s are Runge (up to biholomorphism), or whether all Short $\mathbb{C}^k$'s are Loewner remains unsettled,  we show that the typical Short $\mathbb{C}^k$'s are Loewner. In the final section, we construct a bunch of non-autonomous basins of attraction, which serve as interesting  examples of Short $\mathbb{C}^2$'s. 
\end{abstract}


\maketitle
\no 
\section{Introduction}
\no 
For $k\geq 2$, let  $\Omega \subseteq \mathbb{C}^k$  be an increasing union of unit balls, i.e.,
\[
\Omega_1 \subseteq \Omega_2 \subseteq \cdots \subseteq \Omega_k \subseteq \Omega_{k+1} \subseteq \cdots \subseteq \Omega=\bigcup_{i=1}^\infty \Omega_i,
\]
 where each $\Omega_i$ is biholomorphic  to the unit ball in $\mathbb{C}^k$. 
 \begin{defn}
 An increasing union of unit balls $\Omega \subseteq \mathbb{C}^k$, for $k\geq 2$,  is called Short $\mathbb{C}^k$ if the  Kobayashi metric vanishes identically on $\Omega$ but there exists a non-constant bounded above plurisubharmonic function (psh) defined on $\Omega$.  
 
 \end{defn}
 
Genesis of these Short $\mathbb{C}^k$'s lies in the {\it{union problem}}: Is an increasing union of Stein manifolds always Stein?  
 In \cite{F1}, the first author settled this question by showing that in dimension three onwards, there exist increasing sequences of balls whose final unions are not Stein (also see \cite{Wold1}).  Another related theme is to look for the model domains for increasing unions of balls. The possible model domains for increasing unions of balls  when the Kobayashi metric does not vanish identically therein, was discussed in \cite{FS} (also see \cite{FSt}).
In particular, it follows (from \cite{FS}) that in dimension two, up to biholomorphism any such domain $\Omega$ is either the unit ball or the product domain $\Delta \times \mathbb{C}$, where $\Delta$ is the unit disc.  On the other hand, when the Kobayashi metric on an increasing union of balls $\Omega$ vanishes identically, one can immediately see that $\Omega$ could be all of $\mathbb{C}^k$ or a biholomorphic copy of $\mathbb{C}^k$, the so called Fatou-Bieberbach (FB) domain. Further, in \cite{F}, another possible model domain  for $\Omega$ was found when the Kobayashi metric vanishes identically on $\Omega$: It was shown that $\Omega$ could be a Short $\mathbb{C}^k$ which is evidently never biholomorphic to any FB domain (or $\mathbb{C}^k$). These domains appear naturally in complex dynamics as non-autonomous basins of attraction of automorphisms of $\mathbb{C}^k$. 
 
Although, recently many interesting examples of Short $\mathbb{C}^k$'s  are found (\cite{ATP}, \cite{Bera}, \cite{BPV}), mainly on par with its sibling FB domains,  the general theory of Short $\mathbb{C}^k$'s is not yet well-developed. The present article started off with an aim to recognize the fundamental properties of Short $\mathbb{C}^k$'s. Of our particular interest is to compare the properties of Short $\mathbb{C}^k$'s and $\mathbb{C}^k$ (or FB domains). Being a  biholomorphic copy of $\mathbb{C}^k$,  any FB domain has infinite volume. So a natural question is whether all Short $\mathbb{C}^k$'s have infinite volume or not. During the course of seeking  the answer, we ended up proving a more general result which answers the before-mentioned question in the affirmative. 
\begin{thm} \label{volume}
For $k\geq 2$, let $\{\Omega_i\}_{i\geq 1}$ be a sequence of domains in $\mathbb{C}^k$, each of which is biholomorphic to the unit ball, such that  $\Omega_i \subseteq \Omega_{i+1}$ for $i\geq 1$ and the Kobayashi metric vanishes identically on $\Omega$. Then the volume of $\Omega$ is infinite and the Bergman space 
\[
A^2(\Omega)=\left\{f: \Omega \rightarrow \mathbb{C}: \int_\Omega {\lvert f \rvert}^2 < \infty\right\}
\]
consists only of the function $f\equiv 0$.
\end{thm}
As an immediate corollary we get the following.
\begin{cor}
Short $\mathbb{C}^k$'s have infinite volume. Further the Bergman space of any Short $\mathbb{C}^k$ is trivial. 
\end{cor}
Next we discuss some other intriguing by-products of Theorem \ref{volume}.

\medskip 
\no 
{\it A remark on Bedford conjecture:}
We address Problem 24 in \cite{A}.
Let $\{F_j\}_{j\geq 1}$ be a sequence of automorphisms of $\mathbb{C}^k$ such that $F_j(0)=0$, for all $j\geq 1$. Further, assume that $F_j$'s satisfy {\it uniform bound condition}, i.e., there exist $0<c<d<1$ such that 
\[
c \lVert z\rVert \leq \lVert F_j(z)\rVert \leq d \lVert z\rVert,
\]
for all $z\in B(0,1)$, the unit ball in $\mathbb{C}^k$. Then the {\it Bedford conjecture} (see \cite{A}) states that the non-autonomous basin of attraction 
\[
\Omega_{\{F_j\}}=\{z\in \mathbb{C}^k: F_n \circ \cdots \circ F_1(z)\rightarrow 0 \text{ as } n\rightarrow \infty\}
\]
is biholomorphic to $\mathbb{C}^k$. While  the conjecture still remains open, it is well-known that $\Omega_{\{F_j\}}$ is an increasing union of balls and the Kobayashi metric vanishes identically on $\Omega_{\{F_j\}}$.  Therefore by Theorem \ref{volume},  $\Omega_{\{F_j\}}$ has infinite volume and has trivial Bergman space. 
 
In \cite{Wieg}, it was shown that the dimension of the Bergman space of any domain in $\mathbb C$ is either $0$ or $\infty$. He also showed that the result is not true in higher dimensions by finding explicit examples of domains with finite dimensional Bergman space. However, the domains he constructed are (Reinhardt but) not pseudoconvex.  This triggered the question whether there exists a pseudoconvex domain in higher dimension with finite dimensional Bergman space. This question is still unsettled. However, it was addressed in some recent works (see  \cite{J},
 \cite{GHH}, \cite{PZ}).
In particular, in \cite{GHH} Gallagher--Harz--Herbort gave a sufficient condition for  pseudo-convex domains to have infinite dimensional Bergman spaces in terms of their cores. Further they found three different classes of pseudo-convex domains which satisfy the ``core" condition and thus have infinite dimensional Bergman spaces.
We discuss these domains briefly here:
\begin{itemize}
\item 
Let $\Omega$ be a domain and let $p\in \partial \Omega$ be such that there exists a continuous psh function $\varphi_p$ defined on a one-sided open neighbourhood $\mathcal{N}_p$ such that $\varphi_p^*(z)<0$ for all $z\in \overline{\mathcal{N}}_p \setminus \{p\}$ and $\varphi_p^*(p)=0$ where $\varphi_p^*(z)=\limsup_{z'\rightarrow z} \varphi_p(z')$. Thus $p$ is a local peak point
for the class of continuous psh functions. In this case,  $\dim A^2(\Omega)=\infty$;
\item 
If there exists a point on the boundary of $\Omega$ near which $\partial \Omega$ is $C^\infty$ smooth and of finite type in the sense of D'Angelo, then $\dim A^2(\Omega)=\infty$;
\item 
If there exists a point on the boundary of $\Omega$ near which the domain is strictly pseudo-convex, then $\dim A^2(\Omega)=\infty$.
\end{itemize}
In light of the above discussion, we get:
\begin{cor}
For $k\geq 2$, let $\Omega \subseteq \mathbb{C}^k$ be an increasing union of balls on which the Kobayashi metric vanishes identically (in particular, if $\Omega$ is a Short $\mathbb{C}^k$), then it has neither a local peak point nor any point of finite type on its boundary. Further, $\partial \Omega$ fails to be strictly pseudo-convex at each of its point and $\Omega$ cannot have any strictly pseudo-convex neighbourhood.
\end{cor}

For $k\geq 2$, the primary source of examples of Short $\mathbb{C}^k$'s are the non-autonomous basins of attraction of sequence of automorphisms $F_j$'s of $\mathbb{C}^k$ which are of the form:
\begin{equation} \label{shift}
F_j(z_1, \cdots, z_k):=\left(z_1^d+P_{1,j}(z_1,\ldots, z_k), P_{2,j}(z_1,\ldots, z_k), \ldots, P_{k,j}(z_1,\ldots, z_k)\right),
\end{equation}
where $d\geq 2$ and the degree of $ P_{i,j}$'s are at most $(d-1)$. Further,  the moduli of the coefficients of $P_{i,j}$'s are at most $\eta_{j}=a_{j}^{d^j}$ with $0\leq a_{j+1}\leq a_{j} <1$, for $j\geq 1$. These Short $\mathbb{C}^k$'s turn out to be the $0$-sublevel sets of global psh functions, i.e., 
\[
\Omega_{\{F_j\}}=\{z\in \mathbb{C}^k: \varphi(z)<0\},
\]   
for some psh function $\varphi$ on $\mathbb{C}^k$ which can be obtained by modifying the Green's functions of the sequences of automorphisms $\{F_j\}$'s.  In dimension $2$,  another source of these domains are the sublevel sets of Green functions of H\'{e}non maps which are  the most important class of polynomial automorphisms of $\mathbb{C}^2$.  They are of the form $(z,w) \mapsto (p(z)+\delta w,z)$ where $p$ is a polynomial in $z$ of degree $d\geq 2$ and $\delta \neq 0$ (see \cite{F} for details). Further the Green's function for a H\'{e}non map $H$ is  
\[
G_H^+(z,w)=\frac{1}{d^n} \log^+ \left\lVert H^n(z,w)\right\rVert,
\] 
for $(z,w)\in \mathbb{C}^2$, where $\log^+(x)=\max \{\log x,0\}$. Now both of these classes of examples are Runge, i.e.,  holomorphic functions on these domains can be approximated uniformly on compact sets by polynomials,  simply because they are sublevel sets of global psh functions.
However, there exist examples of Short $\mathbb{C}^k$'s which are not Runge. 
In \cite{Wold}, Wold gave an example of a Fatou Bieberbach domain $U\subseteq \mathbb C^2$ which is not Runge. He showed that there exists a compact set $K\subseteq U$ whose polynomial convex hull contains points outside the domain $U$. Now let $F$ be a biholomorphism between $U$ and $\mathbb{C}^2$.   There are plenty of Short $\mathbb{C}^2$'s, belonging to the large class of before-mentioned examples, each of which contains $F(K)$. Now carrying back these Short $\mathbb{C}^2$'s via $F^{-1}$ we get Short $\mathbb{C}^2$'s in $U$ containing $K$. Clearly these Short $\mathbb{C}^2$'s are non-Runge but biholomorphic to Runge domains. In fact, to best of our knowledge, all known examples of Short $\mathbb{C}^k$'s  are either Runge or biholomorphic to  Runge domains. Rather a bit more can be said: If $\Omega\subseteq \mathbb{C}^k$ is an increasing union of $\Omega_i$'s, for $i\geq 1$ and if each $\Omega_i$ is biholomorphic to the unit ball via a global automorphism of $\mathbb{C}^k$, then by Prop.\ 1.2 in \cite{FrR}, $\Omega$ is Runge (this criterion holds for all known examples of Short $\mathbb{C}^k$'s). These observations prompts the question:  If not Runge, whether a Short $\mathbb{C}^k$ always biholomorphic to a Runge one?  In an attempt to answer this question,
we introduce a sub-class of Short $\mathbb{C}^k$'s called {\it Loewner Short $\mathbb{C}^k$}'s. They are those Short $\mathbb{C}^k$'s which can be parametrized on the positive real axis by strictly monotonically increasing domains each of which is biholomorphically equivalent either to the unit ball or to the unit polydisc. Being a biholomorphic copy of $\mathbb{C}^k$, any FB domain enjoys this kind of holomorphic exhaustion by balls. So while comparing FB domains and Short $\mathbb{C}^k$'s it is natural to be inquisitive about whether Short $\mathbb{C}^k$'s can also be exhausted by a holomorphically varying parametrized family of strictly increasing balls. This serves as another motivation to define Loewner Short $\mathbb{C}^k$'s.
 
 \begin{defn} \label{defn Loewner}
A Short $\mathbb{C}^k$ $\Omega$ is {\it Loewner} if the following holds:
\begin{itemize}
\item[(1)]  $\Omega= \cup_{t\geq 0} \Omega_t$, where either  the domain  $\Omega_t$ is biholomorphic to the unit ball in $\mathbb{C}^k$, for each $t\geq 0$  or for each $t\geq 0$, the domain  $\Omega_t$ is biholomorphic to the unit polydisc in $\mathbb{C}^k$.
\item[(2)] for any $0\leq t<s$, $\Omega_t  \subset \subset \Omega_s$.
\item[(3)] For all $t>0$, $\Omega_t=\cup_{t'<t} \Omega_{t'}$.
\item[(4)] For all $t'<t$, $\overline{\Omega}_{t'}=\cap_{t>t'}\Omega_t$.
\end{itemize}
\end{defn}
\begin{prop}\label{prop Loewner}
Any Loewner Short $\mathbb{C}^k$ is biholomorphic to a Runge domain.
\end{prop}
Using Thm.\ 4.2 in \cite{ABW2} we get that for each $n\ge0$, the pair $\left(\Omega_n, \Omega_{n+1}\right)$ is a Runge pair and then Prop.\ \ref{prop Loewner}  follows using Thm.\ 3.4 in \cite{ABW2}, 
Thus our previously posed question on Runge embedding of Short $\mathbb{C}^k$'s can be modified as follows:

\medskip 
\no 
{\it{Question:}} Is any  Short $\mathbb{C}^k$ Loewner?

\medskip 
 Although we could not settle the question, we identify a large class of Loewner Short $\mathbb{C}^k$'s.

\begin{thm}\label{Loewner 1}
Let $0<a_{n+1} \leq a_n<1$ for all $n\geq 1$.
Let $F_n(z,w)=(z^d+q_n(z)+ \delta_n w,\delta_n z)$  with $1\leq \deg(q_n)\leq (d-1)$ such that the moduli of the coefficients of $q_n$ and $\delta_n$ are bounded by $a_n^{d^n}$.  Then the basin of attraction 
\[
\Omega=\{(z,w): F_n \circ \cdots \circ F_1(z,w)\rightarrow 0\}
\]
is a Loewner Short $\mathbb{C}^2$. 
\end{thm}
From  \cite [Thm.\ 1.4] {F} it follows that $\Omega$ is a Short $\mathbb{C}^k$. Further it follows (from the proof of \cite [Thm.\ 1.4] {F}) that for $r>0$ sufficiently small,  $\Omega= \bigcup_{n=1}^\infty {F(n)}^{-1} \left(\Delta(0;r,r)\right)$, where $F(n)=F_n \circ \cdots \circ F_1$ and $\Delta(0;r,r)=\{(z,w): \lvert z\rvert, \lvert w\rvert<r\}$. Moreover, $\Omega_n={F(n)}^{-1} \left(\Delta(0;r,r)\right) \subseteq {F(n+1)}^{-1} \left(\Delta(0;r,r)\right)=\Omega_{n+1}$, for all $n$ sufficiently large. To prove $\Omega$ to be Loewner it is sufficient to show that $\Omega_n$ can be stretched   to $\Omega_{n+1}$ via a holomorphic family of strictly increasing open sets, each of which biholomorphic to polydisc. In fact, for all $n\geq 1$,  that $F_n(\Delta(0;r,r))$ can be expanded to $\Delta(0;r,r)$ in the above-mentioned way  suffices the purpose. It is straightforward to see that for all $n\geq 1$, $\Delta(0;r,r)$ can be shrunk holomorphically to $F_n(\Delta(0;r,r))$, but a priori it is not clear whether this can be performed in a monotonic manner too. The key step in the proof of Thm.\ \ref{Loewner 1} is to construct a parametrized family of automorphisms $F_{n,t}$ for each $n\geq 1$ and for all $0\leq t \leq 1$ such that $F_{n,0}\equiv F_n$, $F_{n,1}\equiv \rm{Id}$ and $F_{n,t} (\Delta(0;r,r)) \subset \subset  F_{n,s} (\Delta(0;r,r))$, whenever $s>t$. If one looks carefully at the proof of Thm.\ \ref{Loewner 1}, then it becomes apparent that the particular form of $F_n$'s plays the most crucial role in constructing $F_{n,t}$'s with desired properties. Nevertheless, we can stretch the family of $F_n$'s a bit to higher dimensions. 

\medskip 
\no
{\it Let $a_n$'s be as before and let for $k\geq 3$,  $F_n(z_1,z_2, \ldots, z_k)=(z
_1^d+q_n(z_1)+ \eta_n z_k, \eta_n z_1, \ldots, \eta_n z_{k-1})$  with $1\leq \deg(q_n)\leq (d-1)$ such that the moduli of the coefficients of $q_n$ and $\eta_n$ are bounded by $a_n^{d^n}$.  Then the basin of attraction 
\[
\Omega=\{(z,w): F_n \circ \cdots \circ F_1(z,w)\rightarrow 0\}
\]
is a Loewner Short $\mathbb{C}^k$. }

\medskip 
\no 
If we consider any non-autonomous basin of attraction $\Omega$ of polynomial automorphisms of the form (\ref{shift}) (which is known to be Short $\mathbb{C}^k$ from \ \cite[Thm.\ 1.4]{F}), it is not clear how to construct the intermediate automorphisms $F_{n,t}$'s. Nevertheless, we believe that these $\Omega$'s are also Loewner. 

Next we prove that the sublevel sets of Green's functions of H\'{e}non maps are Loewner. As before the form of H\'{e}non maps plays the pivotal role.  That any sublevel set of the Green's function of a H\'{e}non map can be realized as the basin of attraction of a sequence of H\'{e}non maps with rapidly decaying coefficients (in other words,  hypotheses of Thm.\ \ref{Loewner 1} are satisfied) leads us to the following theorem.

\begin{thm}\label{Loewner 2}
Let  $H$ be a H\'{e}non map and let $G_H^+$ be the Green's function of $H$. Then for any $r>0$, the set 
\[
\Omega_r=\{z\in \mathbb{C}^2: G_H^+(z)<r\},
\]
 is a Loewner Short $\mathbb{C}^2$.
\end{thm}

\medskip 
We conclude the article by constructing new examples of Short $\mathbb{C}^2$'s.  
\subsection*{Example 1:} Instead of a single H\'{e}non map,  we consider a random sequence of H\'{e}non maps $\{F_n\}_{n\geq 1}$ such that the moduli of the coefficients of $F_n$'s are uniformly bounded above. Then it follows from  \cite[Prop.\ 1.1]{PV} that the corresponding random Green's function 
\[
G^+_{\{F_n\}}(z,w):=\lim_{n\rightarrow \infty}\frac{1}{d^n} \log^+ \lVert F_n\circ \cdots \circ F_1(z,w)\rVert,
\]
for  $(z,w)\in \mathbb{C}^2$, exists where $d=\deg(F_n)$ for all $n\geq 1$ and $\log^+(x)=\max\{\log x,0\}$.
In spirit of  \cite[Thm.\ 1.12]{F} (which shows that the sublevel sets of the Green's function of a single H\'{e}non map are Short $\mathbb{C}^2$'s), we prove the following:

\begin{thm} \label{sublevel ran Green}
Let $\{F_n\}$ be a sequence of H\'{e}non maps of the form $F_n(z,w)=(z^d+q_n(z)+\delta_n w, z)$ where $q_n$'s are polynomials of degree at most $ (d-1)\geq 1$ and
$\delta_n\neq 0$, for all $n\geq 1$. Further, assume that  the moduli of the coefficients of $q_n$'s and $\delta_n$'s are uniformly bounded above.  Then for any $r>0$, the sublevel set
\[
\Omega_{\{F_n\},r}=\left\{(z,w)\in \mathbb{C}^2: G_{\{F_n\}}^+(z,w)<r \right\}
\]
is a Short $\mathbb{C}^2$.
\end{thm}

\subsection*{Example 2:} For  $0<c<1$, let $F_n(z,w)=(z^2+c^{2^n} w, c^{2^n}  z)$ for all $n\geq 1$.
Then by \cite[Thm.\ 1.4]{F}, it follows that the non-autonomous basin of attraction of $F_n$'s at the origin is always a Short $\mathbb{C}^2$. The following theorem records that if we replace the coefficients of the linear terms of $F_n$'s by $c^{t_n^n}$ with $t_n$ converging to $2$, the corresponding basin of attraction is not necessarily always a Short $\mathbb{C}^2$.  The following theorem can also be compared with \cite[Thm.\ 1.10]{F}.
\begin{thm}\label{SC2 FB}
For $n\geq 1$, let $F_n(z,w)=(z^2+c^{t_n^n} w, c^{t_n^n}  z)$ where $t_n\rightarrow 2$ and $0<c<1$. Let 
\[
\Omega= \left\{(z,w): F_n \circ \cdots \circ F_1(z) \rightarrow 0\right\} 
\]
be the non-autonomous basin of attraction of the sequence of automorphisms $\{F_n\}$ at the origin. Then,  
\begin{enumerate}
\item [A.]
There exists a sequence $t_n \rightarrow 2$ such that $\Omega$ is a Short $\mathbb{C}^2$.
\item[B.] 
There exists a sequence $t_n \rightarrow 2$ such that $\Omega$ is a Fatou-Biberbach domain.\end{enumerate}
\end{thm}

\subsection*{Acknowledgements:} The present article was initiated during the 2020 Complex Dynamics conference at the CIRM-Luminy, France. Both authors are grateful to CIRM for providing local hospitality during the conference. The second author would like to thank Koushik Ramachandran and Sivaguru Ravisankar for partially supporting her travel to CIRM. The second author was supported by National Board of Higher Mathematics postdoctoral fellowship.


\section{Volume of Short $\mathbb{C}^k$'s: Proof of Thm. \ref{volume}}
\begin{proof}
Let $\Omega=\cup_{i=1}^\infty\Omega_i.$ Let $ B(0,1)$ be the unit ball in $\mathbb{C}^2$.  We choose biholomorphisms $\phi_i: B(0,1)\rightarrow \Omega_i$. We can assume that  $0\in \Omega_1$ and hence in all $\Omega_i$'s. After pre-composing with automorphisms of the unit ball, we can assume that $\phi_i(0)=0$ for all $i.$
		
		Pick an $\epsilon>0.$ Let $\xi$ denote a tangent vector at $0$ of length $1.$ Since the Kobayashi metric of $\Omega$ vanishes at $0$ identically, there exists a holomorphic map $\psi:\Delta\rightarrow \Omega$ so that $|\psi'(0)|=\lambda \xi$ for some constant $\lambda> \frac{4}{\epsilon}$, where $\Delta$ is the unit disc in $\mathbb{C}$. Hence there is a holomorhic map $\psi_1:\overline{\Delta}\rightarrow \Omega$ such that $\psi_1'(0)=\lambda \xi/2$. Then the image is a compact subset of some $\Omega_i$.  After small rotation we then know that there is a relatively open neighbourhood of $\xi$ in the unit sphere so that for each $\xi'$ in this neighbourhood there is a map $f:\Delta\rightarrow \Omega_i$ so that $f(0)=0, f'(0)=\mu \xi'$ with $\mu>2/\epsilon.$ This then is true for all larger $i$. Hence by compactness of the unit sphere, there is an integer $j$ so that for any tangent vector $\xi$ at zero of length one, there is a holomorphic map $f:\Delta\rightarrow \Omega_j$ so that $f(0)=0$ and  $f'(0)=\lambda \xi$ for some $\lambda>2/\epsilon.$   This is true for any large $j$.
		
Next consider the composite maps $g=\phi_j^{-1}\circ f.$ These are holomorphic maps from the unit disc to the unit ball. Moreover, $g(0)=0$ and 
$$
g'(0)=(\phi_j^{-1})'(f(0))f'(0)=(\phi_j^{-1})'(f(0))(\lambda \xi).
$$
Then by Schwarz lemma it follows that 
$$
\|(\phi_j^{-1})'(f(0))(\lambda \xi)\|\leq 1, \text{ i.e., } \|(\phi_j^{-1})'(f(0))( \xi)\|\leq 1/\lambda<\epsilon/2.
$$
		
Let $\tilde\xi$ be the unit vector in the direction of $(\phi_j^{-1})'(f(0))( \xi)$. Then $(\phi_j^{-1})'(f(0))( \xi)=\sigma \tilde{\xi}$ and $\sigma<\epsilon/2.$ Therefore $\phi_j'(0)(\tilde{\xi})=\frac{\xi}{\sigma}.$
Thus it follows that $\|\phi_j'(0)(\tilde{\xi})\|\geq \frac{2}{\epsilon}$. But as $\xi$ runs over all unit tangent vectors, the same is true for the vectors $\tilde{\xi}.$ Therefore, we have shown: For any large enough $j,$ we have that if $\tilde{\xi}$ is any unit tangent vector at the origin, then $\|\phi_j'(0)(\tilde{\xi})\|\geq \frac{2}{\epsilon}.$
		
Let $A$ be a 2 by 2 matrix so that $\|A(x)\|\geq 2\|x\|/\epsilon$ for every vector $x.$ Then the inverse matrix $B$ satisfies $\|B(y)\|\leq \epsilon\|y\|/2$. Hence each entry in the matrix of $y$ has size at most $\epsilon/2$. Then the Jacobian of the matrix can be at most $\epsilon^2/2.$ But this implies that the Jacobian of $A$ is at least $2/\epsilon^2.$
		
To finish the proof we only need to show that there is 
no non-trivial $L^2$ holomorphic function on $\Omega.$ In particular this shows that the function  $f\equiv 1$ is not in $L^2(\Omega)$ which implies that  the volume of $\Omega$ is infinite. So assume to the contrary that there is such an $f$ on $\Omega.$ We could have chosen the point $0$ differently, so we can assume that $f(0)=c\neq 0.$ Then for every $i,$ we have that $\int_{\Omega_i} |f|^2<1.$ By the change of variable formula this shows that 
$$
\int_{\mathbb B(0,1)}\left \lvert(f\circ \phi_i)\times {\mbox{Jac}}(\phi_i)\right\rvert^2 dV<1.
$$
Let $g_i=(f\circ \phi_i)\times {\mbox{Jac}}(\phi_i)$. Then $g_i$ is a holomorphic function on the unit ball
with $|g_i(0)|\geq 2|c|/\epsilon^2$ 
		and $\int_{\mathbb B(0,1)}|g_i|^2<1$ for all large enough $i.$
		
	After writing $g_i(z)=g_i(0)+{\rm{h.o.t.}}$, if we estimate the integral of $|g_i|^2$ we get an estimate from below by integrating only the constant term. This gives a contradiction if we choose $\epsilon$ small enough. 
\end{proof}


\section{Loewner Short $\mathbb{C}^k$'s}

\subsection*{Proof of Theorem \ref{Loewner 1}}
By Thm.\ 1.4 in \cite{F}, it follows that $\Omega$ is a Short $\mathbb{C}^2$.  Now we prove that $\Omega$ is Loewner.  
\subsection*{Test Case:}
To make the idea of the proof transparent, we first deal with the simplest case:  Namely, $\Omega$ is non-autonomous basin of attraction of quadratic H\'{e}non maps $F_n$'s where $F_n(z,w)=(z^2+ a_n^{2^n} w,a_n^{2^n} z)$ with $0<a_{n+1} \leq a_n<1$, for all $n\geq 1$. 

For $r>0$ sufficiently small,  let $\Delta^2(0;r,r)$ be the bidisc of radius $r$ and 
$$
\Omega_n=F_1^{-1} \circ \cdots \circ F_n^{-1}(\Delta^2(0;r,r)).
$$ 
It follows from Thm 1.4 in \cite{F} that $\Omega_n \subseteq \Omega_{n+1}$ and  $\Omega= \cup_{n=1}^\infty \Omega_n$. We show that for any $n\geq 1$, $F_n(\Delta^2(0;r,r))$ can be deformed holomorphically and monotonically  to $\Delta^2 (0;r,r)$. 

For $0\leq t \leq 1$, define $\Delta^2_t=\Delta^2\left(0,(1+t)r, (1+2t)r\right)$ and $F_{n,t}$ on $\Delta^2_t$ as follows:  
$$
F_{n,t}(z,w)=((1-t)z^2+[a_n^{2^n}+rt]w,a_n^{2^n}z).
$$
Note that when $t=0,$ then $F_{n,0}\equiv F_n$ and $\Delta^2_0=\Delta^2(0,r,r)$.  For $t=1$,  $F_{n,1}(z,w)= \left((a_n^{2^n}+r)w,a_n^{2^n}z\right)$
 and the image of $\Delta^2_1$ under the map $F_{n,1}$ becomes $\Delta^2 \left(0,(a_n^{2^n}+r)3r,a_n^{2^n}2r\right)$. 
 
 First we show that as $t$ increases, the images of $\Delta^2_t$'s under the map $F_{n,t}$'s strictly increase, i.e.,
 if $0\leq t<s\leq 1,$ then
 \begin{equation}\label{subset}
 F_{n,t}(\Delta^2_t)\subset\subset F_{n,s}(\Delta^2_s).
 \end{equation}
 Fix a $z_0$ with $|z_0|<(1+t)r.$ Now we consider the vertical slice
 $L_{z_0}=\{w\in \mathbb{C}: |w|<(1+2t)r\}$. The image 
 $$
 F_{n,t}(L_{z_0})=\left\{\left((1-t)z_0^2+[a_n^{2^n}+rt]w,a_n^{2^n}z_0\right): |w|<(1+2t)r \right\},
 $$ 
 is the horizontal disc at height $a_n^{2^n}z_0$ with center $(1-t)z_0^2$ and radius $(a_n^{2^n}+rt)(1+2t)r$.
 The image of $L_{z_0}$  under the map $F_{n,s}$ gives the  horizontal disc with center $(1-s)z_0^2$ and radius $(a_n^{2^n}+rs)(1+2s)r$ at the same height. The condition of strict monotonicity is the following:
 $$
 (1-s){\lvert z_0 \rvert}^2+(a_n^{2^n}+rs)(1+2s)r>(1-t){\lvert z_0\rvert }^2+(a_n^{2^n}+rt)(1+2t)r,
 $$
 which is equivalent to
 $$
 (a_n^{2^n}+rs)(1+2s)r-(a_n^{2^n}+rt)(1+2t)r>(1-t){\lvert z_0\rvert}^2-(1-s){\lvert z_0\rvert}^2,
 $$
 i.e. to
 $$
 [a_n^{2^n}(1+2s)r-a_n^{2^n}(1+2t)r]+[r^2s(1+2s)-r^2 t(1+2t)]>(s-t) {\lvert z_0\rvert}^2.
 $$
 Thus it suffices to prove that
 $$
 r^2s(1+2s)-r^2 t(1+2t)>(s-t)|z_0|^2,
 $$
  which is equivalent to 
  $$
  r^2(s-t)+2r^2(s^2-t^2)>(s-t)|z_0|^2,
  $$
  in other words, to
  $$
  r^2(1+2s+2t)>|z_0|^2.
  $$
  Now since $|z_0|<(1+t)r$ and $s>t$, it suffices to prove that $(1+4t)>(1+t)^2$ which is clearly true since $t<1.$ 
  
 For $0\leq t \leq 1$, let $\Omega_t=F_{n,t}(\Delta^2_t)$. We now show that $\Omega_t=\bigcup_{t'<t} \Omega_{t'}$. 
 Since (\ref{subset}) holds,  clearly $\bigcup_{t'<t} \Omega_{t'} \subseteq \Omega_t$.  To prove the reverse containment, let $(z_0,w_0)\in \Omega_t$. Then 
 $$
(z_0,w_0)=F_{n,t}(z,w)=\left((1-t)z^2+[a_n^{2^n}+rt]w,a_n^{2^n}z\right)
$$
for some $(z,w)\in \Delta^2_t$. Thus $z={w_0}/{a_n^{2^n}}$  and $(1-t)z^2+[a_n^{2^n}+rt]w=z_0$. If there exists $t'<t$ such that $F_{n,t'}(z',w')=(z_0,w_0)$ for some $(z',w')\in \Delta^2_{t'}$, then 
\[
(1-t)z^2+[a_n^{2^n}+rt]w=(1-t')z^2+[a_n^{2^n}+rt']w',
\]
i.e.,
\[
(t'-t)z^2+(a_n^{2^n}+rt)w=(a_n^{2^n}+rt')w'
\]
which is equivalent to 
\[
w'=w+\frac{(t-t')}{(a_n^{2^n}+rt')}[rw-z^2].
\]
To check the existence of some $w'$ with $\lvert w'\rvert <(1+2t')r$, it is sufficient to check if $t'<t$ can be chosen so that 
\[
\lvert w\rvert+\frac{(t-t')}{(a_n^{2^n}+rt')}\left[r\lvert w\rvert+{\lvert z\rvert}^2\right]<(1+2t')r,
\]
i.e., 
\begin{equation} \label{tt'}
\lvert w\rvert \left(1+\frac{(t-t')r}{(a_n^{2^n}+rt')} \right)+\frac{(t-t')}{(a_n^{2^n}+rt')} {\lvert z\rvert}^2 < (1+2t')r.
\end{equation}
To check whether $t'<t$ exists so that (\ref{tt'}) holds, it is enough to see whether the  following holds for $t'$ sufficiently close to $t$: 
\begin{equation} \label{tt'1}
\lvert w\rvert \left(1+\frac{t-t'}{a_n^{2^n}} \right)+\frac{(t-t')}{a_n^{2^n}} {\lvert z\rvert}^2 < (1+2t')r.
\end{equation}
Now there exists $\tilde{t}<t$ such that $\lvert w\rvert <(1+2 \tilde{t})r$. Thus it is possible to choose  $t'<t$ such that (\ref{tt'1}) holds. This proves 
\[
\Omega_t=\bigcup_{t'<t} \Omega_{t'}.
\]

Next we prove that $\overline{\Omega}_t = \bigcap_{t'>t} \Omega_{t'}$. That $\overline{\Omega}_t \subseteq  \bigcap_{t'>t} \Omega_{t'}$ follows from (\ref{subset}). Now let $(z_0,w_0)\in \Omega_{t'}$ for all $t'>t$. Therefore, 
\[
(z_0,w_0)=\left((1-t)z_{t'}^2+[a_n^{2^n}+rt]w_{t'},a_n^{2^n}z_{t'}\right)
\]
for some $(z_{t'},w_{t'})\in \Delta^2_{t'}$. This implies $z=z_{t'}={w_0}/{a_n^{2^n}}$ for all $t'>t$ and 
\[
(1-t')\frac{w_0^2}{a_n^{2^{n+1}}}+\left( a_n^{2^n}+rt'\right)w_{t'}=z_0,
\]
i.e.,
\[
w_{t'}=\left[ z_0-(1-t')\frac{w_0^2}{a_n^{2^{n+1}}}\right]\frac{1}{\left( a_n^{2^n}+rt'\right)}
\]
for all $t'>t$. Now clearly $(z,w_t)=\lim_{n\rightarrow \infty} (z,w_{t'}) \in \overline{\Delta_t^2}$  and $F_{n,t}(z,w_t)=(z_0,w_0)$. Therefore $\bigcap_{t'>t} \Omega_{t'} \subseteq \overline{\Omega}_t$.
  
 Till now we proved that for all $n\geq 1$, $\Omega_0=F_n(\Delta^2(0;r,r))$ can be expanded monotonically to $\Omega_1=\Delta^2 \left(0,(a_n^{2^n}+r)3r,a_n^{2^n}2r\right)$ via a holomorphically varying family of domains $\Omega_t$ for $0\leq t \leq 1$ that satisfies the properties enlisted in Definition \ref{defn Loewner}. Now it is easy to see that the bidisc $\Omega_0'=\Delta^2 \left(0,(a_n^{2^n}+r)3r,a_n^{2^n}2r\right)$ can be expanded to the bidisc $\Omega_1'=\Delta^2(0;r,r)$  monotonically  and holomorphically via bidisc $\Omega_t'=\left(0;tr +(1-t)(a_n^{2^n}+r)3r, tr+(1-t)a_n^{2^n}2r\right)$ for $0\leq t' \leq 1$. Further, this family of bidiscs satisfies the properties indicated in Definition \ref{defn Loewner}.  Therefore, for each $n\geq 1$, we showed existence of domains $\{\Omega_{n,t}\}_{0\leq t' \leq 1}$ such that $\Omega_{n,0}=F_n\left(\Delta^2(0;r,r)\right)$ and $\Omega_{n,1}=\Delta^2(0;r,r)$. Further $\Omega_{n,t} \subset \subset \Omega_{n,s}$ if $0\leq t <s \leq 1$, $\Omega_{n,t}=\bigcup_{t'<t} \Omega_{n,t'}$ and $\overline{\Omega}_{n,t} =\bigcap_{t'>t} \Omega_{n,t'}$. In turn for each $n$, we have 
\begin{eqnarray}\label{Omega}
&&\Omega_{n-1}=F_1^{-1} \circ \cdots \circ F_n^{-1}(F_n(\Delta^2(0;r,r)))\subseteq \cdots \subseteq F_1^{-1} \circ \cdots \circ F_n^{-1}(\Omega_{n,t}) \nonumber\\
&\subseteq & 
\cdots \subseteq F_1^{-1} \circ \cdots \circ F_n^{-1}(\Omega_{n,s})
\subseteq \cdots \subseteq F_1^{-1} \circ \cdots \circ F_n^{-1}((\Delta^2(0;r,r)))=\Omega_n.
\end{eqnarray}
This proves $\Omega$ is Loewner when for each $n\geq 1$, $\deg(F_n)=2$ and $a_n$ is real.

\subsection*{General Case}
For each $n\geq 1$, 
$
F_n(z,w)=\left(p_n(z)+\delta_n w, \delta_n z\right),
$
where $p_n(z)=z^d+q_n(z)$  with  $q_n$ a polynomial of degree $d-1$. 
For each $n\geq 1$, we modify $F_n$ as follows. Let $0\leq t_n < 2\pi$ be such that $\delta_n'=e^{it_n}\delta_n >0$. Define
\[
\tilde{F}_n(z,w)=\left(\tilde{p}_n(z)+\delta_n' w, \delta_n' z\right)
\]
where $\tilde{p}_n(z)=p_n\left(\left({\delta_n'}/{\delta_n}\right)z \right)$. 
 As before we can show that for any $n\geq 1$, the bidisc $\Delta^2(0;r,r)$ can be deformed to $\tilde{F}_n(\Delta^2(0;r,r))$.  Choose $M>> 2^d+{d}/{r^d}$. For each $n\geq 1$, define $F_{n,t}$ on $\Delta^2_t=\Delta^2\left(0,(1+t)r, (1+2t)r\right)$ as follows:
\[
\tilde{F}_{n,t}(z,w)=\left((1-t)\tilde{p}_n(z)+[\delta_n' + M r^{d-1}t]w,\delta_n' z \right).
\]
As before, for  fix a $z_0$ with $\lvert z_0\rvert<(1+t)r$,  the condition of strict monotonicity, i.e., for $0\leq t<s \leq 1$,  $\tilde{F}_{n,t}(\Delta^2_t) \subset \subset \tilde{F}_{n,s}(\Delta^2_s)$ is 
\[
\left(\delta_n' + M r^{d-1}t\right)(1+2t)r+(s-t)\lvert \tilde{p}_n(z_0)\rvert <\left(\delta_n' + M r^{d-1}s\right)(1+2s)r,
\]
which is equivalent to showing
\[
\delta_n'\left[(1+2s)-(1+2t)\right]r+Mr^{d-1}s(1+2s)r-Mr^{d-1}t(1+2t)r > (s-t)\lvert \tilde{p}_n(z_0)\rvert.
\]
Thus it is sufficient to show 
\[
Mr^d(s-t)+2Mr^d(s^2-t^2) > (s-t)\lvert \tilde{p}_n(z_0)\rvert,
\]
or equivalently, 
\begin{equation} \label{M}
Mr^d(1+2s+2t)>\lvert \tilde{p}_n(z_0)\rvert.
\end{equation}
Now since $\lvert z_0\rvert<(1+t)r$, $\lvert \tilde{p}_n(z_0)\rvert< {(1+t)}^d r^d+a_n^{2^n}\left({(1+t)}^{d-1} r^{d-1}+\cdots +(1+t)r+1\right)<2^d r^d+d$. Thus clearly (\ref{M}) holds.

Thus $\tilde{F}_{n,1}(\Delta^2(0;r,r))$ and thus $\Delta^2(0;r,r)$ can be deformed holomorphically and monotonically to $\tilde{F}_{n,0}(\Delta^2(0;r,r))$. We claim that $\tilde{F}_{n,0}(\Delta^2(0;r,r))={F}_{n,0}(\Delta^2(0;r,r))$. Fix a  $z_0$ with $\lvert z_0\rvert<r$, then $F_n(z_0,w)=\tilde{F}_n\left(({\delta_n}/{\delta_n'}) z_0, {(\delta_n}/{\delta_n'}) w\right)$. Thus $ {F}_{n,0}(\Delta^2(0;r,r))  \subseteq\tilde{F}_{n,0}(\Delta^2(0;r,r))$. Similarly it follows that $\tilde{F}_{n,0}(\Delta^2(0;r,r))\subseteq {F}_{n,0}(\Delta^2(0;r,r))$.  
The rest of the proof follows as before. 

\subsection*{Proof of Thm.\ \ref{Loewner 2}}

We complete the proof in three steps. 

\no 
{\it Step 1:} Let $\deg(H)=d$ and let $r_n \uparrow r$ as $n \rightarrow \infty$. Let $\delta<1$ and for each $n\geq 1$, let us choose $m(n)\geq n$ such that 
\begin{equation} \label{exp}
e^{\left(r_{n+1}-r_{n+2}\right)d^{m(n)}} < \frac{\delta}{n}.
\end{equation}
Let ${\lVert(z,w)\rVert}_1=\max\{\lvert z\rvert, \lvert w\rvert\}$.
For each $n\geq 1$, set
\[
\Omega_n=\left\{(z,w)\in \mathbb{C}^2: {\lVert H^{m(n)}(z,w)\rVert}_1 < e^{r_{n+1} d^{m(n)}}
\right\}.
\]
Now since (\ref{exp}) holds, 
\begin{eqnarray} \label{log}
&&(r_{n+1}-r_{n+2}) d^{m(n)} < \log\left(\frac{\delta}{n}\right)<-\log n \nonumber \\
&\Rightarrow& (r_{n+2}-r_{n+1}) > \frac{\log n}{d^{m(n)}} \nonumber\\
&\Rightarrow& r_{n+1}+\frac{\log n}{d^{m(n)}} < r_{n+2}.
\end{eqnarray}

\medskip 
\no 
{\it Claim: $\Omega_r$ is increasing union of $\Omega_n$'s}. 

\medskip 
\no 
There exists $L>1$ such that 
\begin{equation}\label{Green est}
G_H^+(z,w) \leq \log^+{ \lVert (z,w)\rVert}_1 +L
\end{equation}
for all $(z,w)\in \mathbb{C}^2$. 
Now 
\[
\frac{1}{d^{m(n)}} \log^+{ \lVert H^{m(n)}(z,w)\rVert}_1 < r_{n+1}
\]
for all $(z,w)\in \Omega_n$.
Thus by (\ref{Green est}), 
\begin{eqnarray}\label{Green est 1} 
&& G_H^+ \circ H^{m(n)}(z,w) \leq \log^+ {\lVert H^{m(n)}(z,w)\rVert}_1 +L \nonumber \\
&\Rightarrow& G_H^+(z,w) \leq \frac{1}{d^{m(n)}} \log^+ {\lVert H^{m(n)}(z,w)\rVert}_1 + \frac{L}{d^{m(n)}}< r_{n+2 } <r. 
\end{eqnarray}
The second last inequality follows from (\ref{log}). 
 We can choose $m(n+1)$ large enough such that 
\[
\frac{1}{d^{m(n+1)}} \log^+ {\lVert H^{m(n+1)}(z,w)\rVert }_1<  \frac{1}{d^{m(n)}} \log^+ {\lVert H^{m(n)}(z,w)\rVert }_1+ \frac{L'}{d^{m(n)}} <r_{n+2}<r
\]
for some $L'>0$ and for all $(z,w)\in \Omega_n$.
Therefore, 
\[
\Omega_1 \subseteq \Omega_2 \subseteq \subset \cdots  =\cup_{n=1}^\infty \Omega_n \subseteq \Omega_r. 
\]
To prove the converse, let $G_H^+(z,w) <r$. This implies $G_H^+(z,w)<r_k$ for some $k\geq 1$. Therefore, 
\[
\frac{1}{d^n} \log^+ {\lVert H^n(z,w)\rVert}_1 <r_k
\]
for all $n\geq n_0$. Therefore, 
$(z,w)\in \Omega_n$ for some large $n$. Thus the claim follows.
 
 \medskip 
 \no 
 { \it Step 2:} Note that for each $n\geq 1$, the global map 
 $\Phi_n =\varphi_{m(n)} \circ H^{m(n)}$ maps  $\Omega_n$ to  $\Delta^2(0;1,1)$  where $\varphi_{m(n)}(z,w)=\left(e^{-r_{n+1}d^{m(n)}}z, e^{-r_{n+1}d^{m(n)} }w\right)$.
 Now we claim that
 \begin{equation*}
 \Omega_r=\{(z,w)\in \mathbb{C}^2: \Phi_n(z,w)\rightarrow 0 \text{ as } n\rightarrow \infty\}.
 \end{equation*}
 For any $(z,w)\in \Omega_{r}$, there exists $n_0$ large enough such that $(z,w)\in \Omega_n$ for all $n\geq n_0$.
 Now if $(z,w)$ is in the non-escaping set $K_H^+=\{(z,w)\in \mathbb{C}^2: \lVert H^n(z,w)\rVert <M \text{ for all } n\geq 1\}$, then clearly, $\Phi_n(z,w) \rightarrow 0$ as $n\rightarrow \infty$. Let $(z,w)\notin K_H^+$ and let $(z,w)\in \Omega_n$ for all $n\geq n_0$, then there exists a large natural number, without loss of generality, $m(n_0)$ say, such that $H^{m(n_0)}(z,w)\in V_R^+$ where $V_R^+=\left\{(z,w)\in \mathbb{C}^2: \lvert z \rvert \geq \max\{\lvert w\rvert,R\}\right\}$. Now
 \begin{align*}
 {\left \lVert H^{m(n)}(z,w)\right\rVert}_1={\left\lVert H^{m(n)-m(n_0)}\left(H^{m(n_0)}(z,w)\right)\right\rVert}_1 \sim {\left\lvert {\left(H^{m(n_0)}(z,w)\right)}_1\right\rvert}^{d^{m(n)-m(n
 _0)}}\\
   <{ \left(e^{r_{n_0+1} d^{m(n_0)} }\right)}^{d^{m(n)-m(n_0)}}=e^{r_{n_0+1}d^{m(n)}}.
 \end{align*}
 Thus 
 \[
\left\lVert \Phi_n(z,w)\right\rVert= \left\lVert\varphi_{m(n)}\circ H^{m(n)}(z,w)\right\rVert <e^{\left(-r_{n+1}+r_{n_0+1}\right)d^{m(n)}} \rightarrow 0 
 \]
 as $n\rightarrow \infty$. Conversely, it is straightforward to see that if $\Phi_n(z,w)\rightarrow 0$ as $n\rightarrow \infty$, then $(z,w)\in \Omega_r$.

Now without loss of generality, we assume that $\Delta^2(0;c,c) \subseteq \Omega_r$, for $c>0$ sufficiently small. Let $\Phi_n(z,w)\in \Delta^2(0;c,c)$. This implies 
\[
{\left\lVert H^{m(n)}(z,w)\right\rVert}_1 < c e^{r_{n+1}d^{m(n)}}.
\]
Therefore combining (\ref{log}) and (\ref{Green est 1}), we get
\begin{eqnarray*}
&& \frac{1}{d^{m(n)}} \log^+ {\left\lVert H^{m(n)}(z,w)\right\rVert}_1 < \frac{\log c}{d^{m(n)}}+r_{n+1}\\
&\Rightarrow& \frac{1}{d^{m(n+1)}} \log^+ {\left\lVert H^{m(n+1)}(z,w)\right\rVert}_1 <\frac{\log c}{d^{m(n)}}+r_{n+1}+\frac{L'}{d^{m(n)}} < \frac{\log c}{d^{m(n+1)}}+r_{n+2}.
\end{eqnarray*}
Thus $\Phi_{n+1}(z,w)\in \Delta^2(0;c,c)$. This shows $\Phi_n^{-1}(\Delta^2(0;c,c)) \subseteq \Phi_{n+1}^{-1}(\Delta^2(0;c,c))$ for all $n$. Also we have $\Omega_r =\bigcup_{n=0}^\infty \Phi_n^{-1}(\Delta^2(0;c,c)) $

 
 \medskip 
 \no 
 {\it Step 3:}
 Without loss of generality we further assume  that $\Delta^2(0;1,1)\subseteq \Omega_r$. Let $\Phi_0 \equiv \rm{Id}$. Then 
 \[
 \Phi_n=(\Phi_n \circ \Phi_{n-1}^{-1})\circ (\Phi_{n-1} \circ \Phi_{n-2}^{-1})\circ \cdots \circ (\Phi_2 \circ \Phi_{1}^{-1}) \circ (\Phi_1\circ \Phi_0^{-1})\circ \Phi_0=F_n \circ \cdots \circ F_1\circ F_0
 \]
 where $F_n=\Phi_n \circ \Phi_{n-1}^{-1}=\varphi_{m(n)} \circ H^{m(n)-m(n-1)}\circ \varphi_{m(n-1)}^{-1}$ for $n\geq 1$ and $F_0\equiv \rm{Id}$. Note that $F(n)^{-1}(\Delta^2(0;c,c)) \subseteq F(n+1)^{-1}(\Delta^2(0;c,c))$ where $F(n)=F_n \circ \cdots \circ F_1\circ F_0$ for all $n$. Further, 
 \[
 \Omega_r= \bigcup_{n=0}^\infty F(n)^{-1}\left(\Delta^2(0;c,c)\right).
 \]
 Now we show that for each $n$, $\Delta^2(0;c,c)$ can be distorted monotonically to $F_n(\Delta^2(0;c,c))$ which in turn proves that $\Omega_r$ is Loewner as in Thm.\ \ref{Loewner 1}.
 For each $n$, consider arbitrary sequence of real numbers satisfying
 \[
 r_n=r_{m(n-1)} < r_{m(n-1)+1} <  r_{m(n-1)+2} < \cdots < r_{m(n)-1}<r_{m(n)}=r_{n+1}.
 \]
Further let
 \[
 \varphi_{m(n-1)+k}(z,w)=\left(e^{-r_{[m(n-1)+k]}d^{[m(n-1)+k]}}z, e^{-r_{[m(n-1)+k]}d^{[m(n-1)+k]}}w\right)
 \]
Now note that 
 \begin{eqnarray*}
 && \varphi_{m(n)} \circ H^{m(n)-m(n-1)}\circ \varphi_{m(n-1)}^{-1} \\ 
 &=& \left[\varphi_{m(n)} \circ H \circ \varphi_{m(n)-1}^{-1} \right]\circ \left[\varphi_{m(n)-1} \circ H \circ \varphi_{m(n)-2}^{-1} \right]\circ \cdots \circ \left[\varphi_{m(n-1)+1} \circ H \circ \varphi_{m(n-1)}^{-1} \right]\\
 &=& L_{m(n)}\circ  L_{m(n)-1} \circ \cdots \circ L_{m(n-1)+1}.
 \end{eqnarray*}
 Now as in Thm.\ \ref{Loewner 1}, one can show that $\Delta^2(0;c,c)$ can be distorted monotonically to $L_{j}(\Delta^2(0;c,c))$'s and therefore, $\Delta^2(0;c,c)$ can be distorted monotonically to $ \varphi_{m(n)} \circ H^{m(n)-m(n-1)}\circ \varphi_{m(n-1)}^{-1}(\Delta^2(0;c,c))=F_n(\Delta^2(0;c,c))$. This completes the proof.
 
 



\section{Examples of Short $\mathbb{C}^2$'s}
\subsection*{Proof of Theorem \ref{sublevel ran Green}}
The proof follows a similar line of arguments as in Thm.\ 1.12 in \cite{F}. We start the proof with a claim.
 
 \no 
{\it Claim:} For any compact set $K\subseteq \Omega_{\{F_n\},r} $   and for any $\epsilon>0$, there exist an open set $U\subseteq \Omega_{\{F_n\},r}$ and an automorphism $\Phi$ of $\mathbb{C}^2$ such that $\Phi(U)= B(0,1)$ and $\Phi(K)\subseteq B(0;\epsilon)$ where $B(0;1)$ and $B(0;\epsilon)$ are balls with center at the origin and of radius $1$ and $\epsilon$, respectively.

We denote the random Green's function corresponding to the sequence of H\'{e}non maps $\{F_n\}_{n\geq m}$ by $G^+_{\{F_n\},m}$ and following the notation in the introduction of this article we have $G^+_{\{F_n\},1} \equiv G^+_{\{F_n\}} $.  By Prop.\ 1.1 in \cite{PV}, we have 
\begin{equation}\label{PV Green}
G^+_{\{F_n\},m+1}\circ F_m \circ \cdots \circ F_1 (z,w)=d^m G_{\{F_n\}}^+(z,w),
\end{equation}
for all $(z,w)\in \mathbb{C}^2$ and for all $m\geq 0$. Further, there exists $L>1$ such that 
\begin{equation}\label{PV est}
G^+_{\{F_n\},m}(z,w) \leq  \log^+ \lVert (z,w) \rVert + L,
\end{equation}
for all $(z,w)\in \mathbb{C}^2$ and for all $m\geq 1$.
Now choose $0<r_1< r_2 <r$ such that $K \subseteq \Omega_{\{F_n\},r_1}$. Further, choose $n$ sufficiently large such that $e^{(r_1-r_2)d^n}<\epsilon$ and set
\[
U=\left\{(z,w)\in \mathbb{C}^2: (z,w)\in F_1^{-1}\circ \cdots \circ F_n^{-1}\left(B\left(0,e^{r_2 d^n}\right)\right)\right\} \subseteq \Omega_{\{F_n\},r} .
\]
This choice is possible since (\ref{PV Green}) and (\ref{PV est}) hold.

Let $\varphi: (z,w)\mapsto \left({z}/{e^{r_2 d^n}},{w}/{e^{r_2 d^n}} \right)$ for $(z,w)\in \mathbb{C}^2$ and let $\Phi=\varphi \circ F_n \circ \cdots \circ F_1$. Then  $\Phi (U)=B(0,1)$ and $\Phi(K)\subseteq B(0,\epsilon)$. 

Let $r_n\uparrow r$  and $R_n \rightarrow \infty$, as $n\rightarrow \infty$. Now $\Omega_{\{F_n\},r} =\bigcup_{n\geq 1} K_n$ where $K_n={\overline{\Omega}}_{r_n} \cap V_{R_n}$ and $V_{R_n}=\{(z,w)\in \mathbb{C}^2: \lvert z\rvert, \lvert w\rvert\}\leq R_n \}$.  For each $n\geq 1$, there exists $m(n)>1$ such that  $e^{(r_{n+1}-r_{n+2})d^{m(n)}}< {1}/{n}$ and 
\[
U_n=\left\{(z,w)\in \mathbb{C}^2: (z,w)\in F_1^{-1}\circ \cdots \circ F_{m(n)}^{-1}\left(B\left(0,e^{r_{n+2} d^{m(n)}}\right)\right)\right\} \subseteq \Omega_{\{F_n\},r}  .
\]
Thus there exists a sequence of automorphisms  $\Phi_n=\varphi_{m(n)} \circ F_{m(n)} \circ \cdots \circ F_1$  such that 
\[
\Phi_n (K_n) \subseteq B\left (0, \frac{1}{n}\right ) \text { and } \Phi_n (U_n)=B(0,1).
\]
where  $\varphi_{m(n)}(z,w)= \left({z}/{e^{r_{n+2} d^{m(n)}}}, {w}/{e^{r_{n+2} d^{m(n)}}} \right) $. 

Let $p\in \Omega_{\{F_n\},r}$ and $\zeta$ be a tangent vector at $p$. Without loss of generality, assume $p=0$. Let $p\in U_{n}$ for all $n\geq n_0$. Now $\Phi_n(p)=0$ for all $n$. Let $(F_1^{-1}\circ \cdots \circ F_{m(n_0)}^{-1})'(0)(\eta_0)=\zeta$. Now $(F_1^{-1}\circ \cdots \circ F_{m(n)}^{-1})'(0)(\eta_n)=\zeta$ where $\eta_n=(F_{m(n)}\circ \cdots \circ F_{m(n_0+1)})'(0)(\eta_0)$. Clearly, $\lvert \eta_n\rvert\leq \lvert \eta_0 \rvert$ for all $n\geq n_0$.
So for any $R>0$, we can choose $n$ sufficiently large such that the map $\rho(\Delta)\subseteq B(0,e^{r_{n+2} d^n})$ where $\rho(z)=R \eta_n z $ and $\Delta$ is the unit disc. 
Now consider the map $T_n=F_1^{-1}\circ \cdots \circ F_{m(n)}^{-1}\circ \rho: \Delta \rightarrow \Omega_r$. Note that $T_n(0)=0$ and $T_n'(0)=R \zeta$. Thus the Kobayashi metric vanishes identically on $\Omega_{\{F_n\},r}$. This completes the proof.

\subsection{Proof of Theorem \ref{SC2 FB}:} We first show the existence of a sequence $t_n \rightarrow 2$ such that the corresponding basin of attraction is a Short $\mathbb{C}^2$.

{\subsubsection*{\textbf{Proof of A}} Let $0<r'<r<1$ and $r_l=r{(r')}^l$. Then 
\begin{equation} \label{Delta}
F_{n+l}(\Delta^2(0;r_l)) \subseteq \Delta^2 (0;r_l^2+c^{t_{n+l}^{n+l}}).
\end{equation}
\no 
{\it Claim:} For large $n$, $t_n^l > \frac{l+1}{2}$, for all $l\geq 1$.  

Let $n$ be so large such that $t_n > 1.9$, then for $l=1$, the claim follows. Assuming that it holds for $l$, we shall prove that it holds for $l+1$. Note that
\begin{equation*}
t_n^{l+1} = t_n^l t_n  > \frac{l+1}{2} 1.9= 1.9 \frac{l}{2}+ \frac{1.9}{2}>\frac{l}{2}+\frac{.9l}{2}+ \frac{1.9}{2}\geq \frac{l}{2}+\frac{2.8}{2} > \frac{l+2}{2}.
\end{equation*}
Therefore, 
\begin{eqnarray*}
\log c^{t_{n+l}^{n+l}} &=& t_{n+l}^{n+l}  \log c  \\
&=& ({t_{n+l}^n})( {t_{n+l}^l}) \log c
\leq  ({t_{n+l}^n}) \frac{(l+1)}{2} \log c \\
&\leq& ({t_{n+l}^n}) \frac{(l+1)}{4} \log c + ({t_{n+l}^n}) \frac{(l+1)}{4} \log c\\
&\leq& \log r(1-r) + (l+1) \log r'
\end{eqnarray*}
for $n$ sufficiently large and for all $l\geq 1$. Thus 
\[
c^{t_{n+l}^{n+l}} \leq r(1-r)r'^{(l+1)}\leq r {(r')}^{l+1}-r^2 {(r')}^{l+1} \leq r {(r')}^{l+1}-r^2 {(r')}^{2l} 
\]
which gives 
\[
r_l^2+ c^{t_{n+l}^{n+l}} \leq r_{l+1}.
\]
This proves that for large $n$,
\begin{equation} \label{Delta est}
F_{n+l}(\Delta^2(0;r_l)) \subseteq \Delta^2 (0;r_{l+1})
\end{equation}
for all $l\geq 1$.
Once we prove (\ref{Delta est}), that $\Omega$ is an increasing union of balls and the Kobayashi metric on $\Omega$ vanishes identically, follow exactly in the same way as in Thm.\ 1.4 in \cite{F}. 

Choose a sequence $\{t_n\}_{n\geq 1}$ such that  
\begin{equation} \label{Tn}
t_n^n \geq 2 t_{n-1}^{n-1}.
\end{equation}
One such example is given by  $t_n=2^{\left(1-\frac{\epsilon_n}{n}\right)}$
with $0<\epsilon_{n+1} \leq \epsilon_{n}< 1$, for all $n\geq 1$. Now let  $F(n)=F_n \circ \cdots \circ F_1$ for all $n\geq 1$
and
\[
\varphi_n(z,w)=\max \{\lvert {(F(n))}_1(z,w)\rvert,\lvert {(F(n))}_2 (z,w)\rvert, c^{t_{n-1}^{n-1}} \}.
\]
If $\varphi_n(z)\leq 1$, then
\begin{eqnarray*}\label{est psh}
\varphi_{n+1}(z,w)&=&\max \left\{{({(F(n))}_1(z,w))}^2+c^{t_{n}^{n}}{({(F(n))}_2(z,w))}, c^{t_{n}^{n}}{({(F(n))}_1(z,w))}, c^{t_{n}^{n}}\right\}\nonumber\\
&\leq& \max \left\{\varphi_n^2(z,w)+c^{t_{n}^{n}} \varphi_n(z,w), c^{t_{n}^{n}}\varphi_n(z,w),c^{t_{n}^{n}}\right\}.
\end{eqnarray*}
Therefore it follows from (\ref{Tn}) that if $\varphi_n(z,w)\leq 1$, then 
\[
\varphi_{n+1}(z,w) \leq 2 \varphi_n^2(z,w).
\]
If $\varphi_n(z,w)\geq 1$, then it is easy to see that
\[
\varphi_{n+1}(z,w) \leq 2 \varphi_n^2(z,w).
\]
Therefore we have 
\[
\varphi_{n+1}(z,w) \leq 2 \varphi_n^2(z,w).
\]
for all $(z,w)\in \mathbb{C}^2$ and 
\[
\frac{1}{2^{n+1}} \log \varphi_{n+1}(z,w) \leq \frac{\log 2}{2^{n+1}}+ \frac{1}{2^{n}} \log \varphi_{n}(z,w). 
\]
Thus the sequence of psh function 
\[
\frac{1}{2^{n}} \log \varphi_{n}(z,w) + \sum_{j>n }\frac{\log 2}{2^{j}}
\]
monotonically decreases to a psh function $\varphi$  in $\mathbb{C}^2$. 

It can be easily checked that if $\varphi(z,w)<0$ for any $(z,w)\in \mathbb{C}^2$, then $(z,w)\in \Omega$. Now  if $(z,w)\in \Omega$, then $(F_n \circ \cdots \circ F_1)(z,w)\rightarrow 0$ as $n\rightarrow \infty$ and thus 
\begin{equation*}
\frac{1}{2^n} \log \varphi_n(z,w) = \max \left\{ \frac{1}{2^n} \log \left\lvert F_1^n(z,w)\right\rvert,  \frac{1}{2^n}\log\left\lvert F_2^n(z,w)\right\rvert, \frac{t_{n-1}^{n-1}}{2^n}\log c\right\}
\leq \frac{1}{2^n} \log c
\end{equation*}
for sufficiently large $n$. This implies 
\[
\frac{1}{2^n} \log \varphi_n(z,w)+\sum_{j>n }\frac{\log 2}{2^{j}} \leq \frac{1}{2^n} \log c + \frac{1}{2^n} \log 2=\frac{1}{2^n} \log(2c) <0
\]
for large $n$. Therefore, $\varphi(z,w)<0$ for all $(z,w)\in \Omega$. It remains to show that $\varphi$ is non-constant on $\Omega$ which follows applying the same argument as in Thm.\ 1.4 in \cite{F}. Thus we omit the proof here.

{\subsubsection*{\textbf{Proof of B}} We prove a series of lemmas to show the existence of a sequence $\{t_n\} \rightarrow 2 $ such that the corresponding non-autonomous basin of attraction turns out to be a Fatou-Bieberbach domain. 

We use the following notations:  For each $n\geq 1$, let $F(n)=F_n\circ \cdots \circ F_1$. For $(z,w)\in \Omega$, let $(z_n,w_n)=F(n)(z,w)$. For a  compact set $K$ in the basin of the attraction $\Omega$, set $\delta_n=\max \left\{|z_n|,|w_n|:(z,w)\in K \right\}$.

\begin{lem}\label{FB1}
Assume that $a_n=c^{t^n}$ for $0<c<1$ and $1<t<2.$ Fix $0<b<1.$
Set $F_n(z,w)=(z^2+a_nw,a_n z).$ Then there exists an $n_0$ such that
 $\delta_{n+k}\leq a_{n+k+1}b^k$ for all $n\geq n_0$ and for all $k\geq 0$.
 \end{lem}

\begin{proof}
There are two steps. First we prove that there exists an arbitrarily large $n,$ so that $\delta_n\leq a_{n+1}.$
Suppose on the contrary there is some $n_0$ so that $\delta_n>a_{n+1}$ for all $n>n_0.$
Then for all such $n$'s, $\delta_{n+1}\leq \delta_n^2+a_{n+1}\delta_n\leq 2\delta_n^2.$
Hence $\delta_{n+k}\leq 2^{1+2+\cdots +2^{k-1}}(\delta_n)^{2^k}.$
So $a_{n+k+1}\leq \delta_{n+k}\leq (2\delta_n)^{2^k}$. This implies
that
$(c^{t^{n+1}})^{t^k}\leq (2\delta_n)^{2^k}$. But we can assume that $2\delta_n<1.$ Then this will fail for all large $k$'s.
The second step is to fix a large such n and then prove that $\delta_{n+k}\leq a_{n+k+1}b^k$ for all positive $k$.
This is true for $k=0.$ We will show that this inquality holds inductively.
So suppose that $\delta_{n+k}\leq a_{n+k+1}b^k$.
Then we have that $\delta_{n+k+1}\leq (a_{n+k+1})^2(b^{2k}+b^k).$
Hence we would like to say that $(a_{n+k+1})^2(b^{2k}+b^k)\leq a_{n+k+2}b^{k+1}.$
This means that
$$
\left( c^{t^{n+k+1}} \right)^{2-t}\leq \frac{b}{b^{k}+1}.
$$
This works as long as we have fixed a large $n$.
\end{proof}
In the next lemma, we assume $1<t_n<2$ for all $n$.
\begin{lem} \label{FB2}
We set $a_n=c^{t_n^n}.$ If there exists an arbitrarily large $n\geq 1$ and a $k(n)\geq1$ (depending on $n$) so that $t_{n+k(n)+1} \leq  2^{\frac{k(n)}{n+k(n)+1}}$, then there exists an arbitrarily large $n$ so that $\delta_n\leq a_{n+1}$.
\end{lem}
\begin{proof}
If the claim of the lemma does not hold, then $\delta_n>a_{n+1}$ for all large $n.$
Therefore, $\delta_{n+1}\leq\delta_n^2+a_{n+1}\delta_n\leq 2\delta_n^2$ and 
inductively we get $\delta_{n+k}\leq 2^{1+2+\cdots+2^{k-1}}\delta_n^{2^k}$ for all $k\geq 1$.
Hence $a_{n+k+1}<(2\delta_n)^{2^k}$ and consequently, $c^{t_{n+k+1}^{n+k+1}}<(2\delta_n)^{2^k}$.
If $n$ is large enough, then $2\delta_n<c.$ Thus $t_{n+k+1}^{n+k+1}>2^k$ for all large $n$ and for all $k\geq 1$. Hence we have $2^{k(n)}< t_{n+k+1}^{n+k+1}<2^{k(n)}$ which is contradiction. This completes the proof.
\end{proof}

\begin{lem} \label{FB3}
Let $0<b, c<1.$ Set $t_n=2^{1-\frac{\epsilon_n}{n}}$ where $\epsilon_{n+1}\geq \epsilon_n+\frac{1}{2^{n/2}\ln 2}$ and $\epsilon_n/n\rightarrow 0$. Let $a_n=c^{t_n^n}$ and let $F_n=(z^2+a_nw,a_n z)$. Suppose there exists an arbitrarily large $n$ and $k(n)\geq 1$ so that $t_{n+k(n)+1}\leq2^{\frac{k(n)}{n+k(n)+1}}$ (This means that $\epsilon_{n+k(n)+1}\geq n+1)$. Then there exists an arbitrarily large $n$ so that for all $k\geq 0$, $\delta_{n+k}\leq a_{n+k+1}b^k$. 
\end{lem}

\begin{proof}
By Lemma \ref{FB2}, there exists an arbitrarily large $n$ so that $\delta_n\leq a_{n+1}.$ Pick such a large $n.$ As in the proof of Lemma \ref{FB1}, we would like to have the estimate
	$(a_{n+k+1})^2 (b^{2k}+b^k)\leq a_{n+k+2}b^{k+1}$
	to prove the lemma for all $k.$
	This means that we want
	$$
	c^{2[2^{(1-\frac{\epsilon_{n+k+1}}{n+k+1})}]^{n+k+1}}
	(b^{2k}+b^k) \leq c^{[2^{(1-\frac{\epsilon_{n+k+2}}{n+k+2}}]^{n+k+2}}b^{k+1},
	$$
	i.e., 
	$$
	c^{2^{1+n+k+1-\epsilon_{n+k+1}}}\leq c^{2^{n+k+2-\epsilon_{n+k+2}}}{\frac{b}{1+b^k}},
	$$
	i.e.,
	$$
	c^{\left(2^{n+k+2}[2^{-\epsilon_{n+k+1}}-2^{-\epsilon_{n+k+2}}] \right)}\leq \frac{b}{1+b^k},
	$$
	i.e,.
$$
c^{2^{n+k+2}\cdot 2^{-\epsilon_{n+k+2}}[2^{\epsilon_{n+k+2}-\epsilon_{n+k+1}}-1]}\leq \frac{b}{1+b^k}.
$$
The above follows if
$$
c^{2^{n+k+2}\cdot 2^{-\epsilon_{n+k+2}}[\ln 2 (\epsilon_{n+k+2}-\epsilon_{n+k+1})]}\leq \frac{b}{1+b^k},
$$
which again follows if
$$
c^{2^{n+k+2}\cdot 2^{-\epsilon_{n+k+2}}\frac{1}{2^{(n+k+1)/2}}}\leq \frac{b}{1+b^k}.
$$	
The above follows if $n$ is large enough. 
\end{proof}
We continue to use previous notations  and suppose that $\epsilon_n\rightarrow \infty.$ Then we prove that $\Omega$ is a Fatou-Bieberbach domain.

Since $\epsilon_n\rightarrow \infty,$ for every $n\geq 1$ there exists an  integer $k(n)\geq 1$ so that $\epsilon_{n+k(n)+1}\geq n+1.$  Then applying Lemma \ref{FB3}, we have an arbitrarily large $n$ so that for all $k\geq 0$, 
$\delta_{n+k}\leq a_{n+k+1}b^k$. We then need to show that  $a_1\cdots a_n\leq c^{n/2} a_{n+1}$ for all large $n$. This is equivalent to the inequality
$$
t_1^1+t_2^2+\cdots +t_n^n\geq n/2+t_{n+1}^{n+1}.
$$
We write this as $ls_n\geq rs_n.$
Here the nth left side is $ls_n= t_1^1+t_2^2+\cdots+t_n^n$ and the nth right side is 
$rs_n=\frac{n}{2}+t_{n+1}^{n+1}.$ 
We compare the growth of the left side and the right side:
$$
G_n=(ls_{n+1}-ls_n)-(rs_{n+1}-rs_n)=(t_{n+1}^{n+1})-(1/2+t_{n+2}^{n+2}-t_{n+1}^{n+1})=2t_{n+1}^{n+1}-t_{n+2}^{n+2}-1/2.
$$
Thus	
\begin{align*}
G_n=2 (2^{1-\frac{\epsilon_{n+1}}{n+1}})^{n+1}-(2^{1-\frac{\epsilon_{n+2}}{n+2}})^{n+2}-1/2
=2^{n+2}(2^{-\epsilon_{n+1}}-2^{-\epsilon_{n+2}})-1/2\\
=2^{n+2}2^{-\epsilon_{n+2}}
	(2^{\epsilon_{n+2}-\epsilon_{n+1}}-1)-1/2.
\end{align*}
Hence
$$
G_n\geq 2^{n+2}2^{-\epsilon_{n+2}}\ln 2(\epsilon_{n+2}-\epsilon_{n+1})-1/2
$$
which implies
$$
G_n\geq 2^{3(n+2)/4}\frac{1}{2^{(n+1)/2}}-1/2\rightarrow \infty.
$$
Now this is the first step in the proof of Lemma 3.13 in \cite{F}. Now the rest of the proof of Lemma 3.13 goes through and then to show that $\Omega$ is a Fatou-Bieberbach domain,  one can follow the same steps as in the the last paragraph of \cite{F}.

\end{document}